\documentclass[12pt]{article}

\newcommand{\media}{\m\kern12mu\hbox{\vrule height9pt depth-3.2pt width5pt} \m\kern-20mu\int}

\newcommand{\m }{\mu }

\usepackage{amsmath,amsthm,amsfonts,amssymb,latexsym}
\usepackage{amsmath}
\usepackage{geometry}
\newtheorem{thm}{Theorem}[section]

\newtheorem{lem}[thm]{Lemma}
\newtheorem{prop}[thm]{Proposition}

\newtheorem{rmk}[thm]{Remark}

\newtheorem{theorem}{Theorem}[section]

\newtheorem{lemma}[theorem]{Lemma}

\begin{document}
\title{On the third cohomology of the Lie algebra of vector fields on weighted densities on $\mathbb{R}$ }

{\author{Wajd Afsi \thanks{D\'epartement de Math\'ematiques, Facult\'e des Sciences, Universit\'e de Sfax, Route de
Soukra, km 3.5, BP 1171, 3000 Sfax, Tunisie.~~~~E.mail: ~~wajd.afsi7@gmail.com
}\and
Salem Omri \thanks{D\'epartement de Math\'ematiques, Facult\'e des Sciences de
Gafsa, Zarroug 2112, Tunisie.~~E.mail: omrisalem991@gmail.com
}}}\maketitle

\begin{abstract}
Let Vect($\mathbb{R}$) be the Lie algebra of smooth vector fields on $\mathbb{R}$ and $\mathbb{F}_{\lambda}$ be the space of $\lambda$-densities on $\mathbb{R}$.
 Vect($\mathbb{R}$) acts on $\mathbb{F}_{\lambda}$ by Lie derivative. In this paper, we compute the third differential cohomology of
 the Lie algebra Vect($\mathbb{R}$) with coeffcients in the space $\mathbb{F}_{\lambda}.$ Explicit cocycles spanning these cohomology spaces are given.

\end{abstract}

\maketitle {\bf Mathematics Subject Classification} (2020). 17B56, 53D55, 58H15.

$\mathbf{Key words}$: Cohomology, Lie algebra, Representation of Lie algebra, Weighted densities.

\numberwithin{equation}{section}
%%%%%%%%%%%%%%%%%%%%%%%%%%%%%%%%%%%%%%%%%%%%%%
\section{Introduction}
%%%%%%%%%%%%%%%%%%%%%%%%%%%%%%%%%%%%%%%%%%%%%%%
Let $\mathrm{Vect}(\mathbb{R})$ be the Lie algebra of all vector fields $g \frac{d}{dx}$ on $\mathbb{R},$
 $g \in \mathcal{C}^\infty(\mathbb{R})$. For any $\lambda \in \mathbb{R}$ we define a structure of $\mathrm{Vect}(\mathbb{R})$-module over $\mathcal{C}^\infty(\mathbb{R})$ by \\
\begin{equation*}\mathbb{L}_{g\frac{d}{dx}}^\lambda(f)= gf'+\lambda g'f,\end{equation*}
where $g,~f \in \mathcal{C}^\infty(\mathbb{R})$ and $g' :=\frac{dg}{dx}$. Denote by $\mathbb{F}_{\lambda}$ the $\mathrm{Vect}(\mathbb{R})$-module structure on
$\mathcal{C}^\infty(\mathbb{R})$ defined by $\mathbb{L}^\lambda$ for a fixed $\lambda$. Geometrically, $\mathbb{F}_{\lambda} = \{fdx^\lambda|~f \in\mathcal{C}^\infty(\mathbb{R})\}$
is the space of weighted densities of weight $\lambda \in \mathbb{R}$.
The space $\mathbb{F}_{\lambda} $ coincides with the space of vector fields, functions and differential $1$-forms for $\lambda = -1,~ 0$ and $1,$ respectively. Obviously the adjoint
$\mathrm{Vect}(\mathbb{R})$-module is isomorphic to $\mathbb{F}_{-1}$. The Lie algebra $\mathrm{Vect}(\mathbb{R})$ has a Lie subalgebra $\mathfrak{aff}(1)$=Span($\frac{d}{dx},x\frac{d}{dx}$).
\vskip0.2cm

In \cite{ksw}, the authors found the first cohomology of $\frak{aff}(1)$ acting on $n$-aray linear differential operators $\mathrm{H}^1_{ \mathrm{diff}}(\mathfrak{aff}(1),\mathcal{D}_{\lambda_1,\cdots,\lambda_n;\mu})$. In \cite{GF,Fu, RO, OT}, the authors found the cohomology of the Lie algebra of vector fields on the circle,
and the analogue super structures:
$\mathrm{H}^1_{\mathrm{diff}}(\mathcal{K}(n),\mathfrak{F}^n_{\lambda})$, where $\mathcal{K}(n)$
is realized as the Lie superalgebra of ${\mathrm Vect}(\mathbb{R})$  \cite{is}.

 In \cite{bou,s}, the author found Cohomology of the vector fields Lie algebras on $\mathbb{R}\mathbb{P}^1$ acting on bilinear differential operators
  and the $\frak{sl}(2)$-relative cohomology of the Lie algebra of  vector fields and differential operators.

Our purpose in this paper is to compute the third differential cohomology from the Lie algebra of smooth vectors fields ${\rm Vect}(\mathbb{R})$ with coefficients
in the space of tensor densities $\mathbb{F}_{\lambda},$ where only cochains given by differential operators are considered.

%%%%%%%%%%%%%%%%%%%%%%%%%%%%%%%%%%%%%%%%%%%%%%
\section{Relative cohomology}
%%%%%%%%%%%%%%%%%%%%%%%%%%%%%%%%%%%%%%%%%%%%%%%
Let us first recall some fundamental concepts from cohomology theory. Let $\mathfrak{g}$ be a Lie algebra acting on a vector space $\mathrm{V}$ and let $\mathfrak{h}$ be a subalgebra
of $\mathfrak{g}$ \cite{Fu,hm}. (If $\mathfrak{h}$ is omitted it is assumed to be $\{0\}$). The space of $\mathfrak{g}$-relative $n$-cochains of $\mathfrak{g}$ with values in $\mathrm{V}$ is the
$\mathfrak{g}$-module
$$C^{n}(\mathfrak{g}, \mathfrak{h}; \mathrm{V})= \mathrm{Hom}_{\mathfrak{h}}(\Lambda^{n}(\mathfrak{g}/\mathfrak{h}); \mathrm{V}).$$
The coboundary operator $\delta_{n} : C^{n}(\mathfrak{g}, \mathfrak{h}; \mathrm{V}) \rightarrow C^{n+1}(\mathfrak{g}, \mathfrak{h}; \mathrm{V})$ is a $\mathfrak{g}$-map satisfying
$\delta_{n} \circ \delta_{n-1} =0$. The kernel of $\delta_{n}$, denoted $\mathrm{Z}^{n}(\mathfrak{g}, \mathfrak{h}; \mathrm{V})$ is the space of $\mathfrak{h}$-relative $n$-cocycles, among them,
the elements in the range of $\delta_{n-1}$ are called $\mathfrak{h}$-relative $n$-coboundaries. We denote by $\mathrm{B}^{n}(\mathfrak{g}, \mathfrak{h}; \mathrm{V})$ the space of $n$-coboundaries.
By definition, the $n^{th}$ $\mathfrak{h}$-relative cohomolgy space is the quotient space
$$\mathrm{H}^{n}(\mathfrak{g}, \mathfrak{h}; \mathrm{V}) = \mathrm{Z}^{n}(\mathfrak{g}, \mathfrak{h}; \mathrm{V}) / \mathrm{B}^{n}(\mathfrak{g}, \mathfrak{h}; \mathrm{V}).$$
We will only need the formula of $\delta_{n}$ (which will be simply denoted $\delta$) in degrees "$0$; $1$;" $2$ and $3$:
\begin{itemize}
\item For $v \in C^{0}(\mathfrak{g}, \mathfrak{h}; \mathrm{V}) = V^{\mathfrak{h}}$, $\delta v (g) :=g\cdot v,$ where $$V^{\mathfrak{h}}=\{ v\in V ~| ~h \cdot v=0~~ \hbox{for all}~ h\in \mathfrak{h}\}.$$
\item For $\Upsilon \in C^{1}(\mathfrak{g}, \mathfrak{h}; \mathrm{V}),$ $$\delta(\Upsilon)(g, h) := g\cdot \Upsilon(h) -h\cdot \Upsilon(g)- \Upsilon([g, h]) ~~\hbox{for any}~~ g,h \in \mathfrak{g}.$$
\item For $\Omega \in C^{2}(\mathfrak{g}, \mathfrak{h}; \mathrm{V})$, $$\delta(\Omega)(x,y,z) := x.\Omega(y,z) -y.\Omega(x,z) +z.\Omega(x,y) -\Omega([x,y],z)+\Omega([x,z],y)-\Omega([y,z],x)$$ for any $x, y, z \in \mathfrak{g}.$
\item For $\Phi \in C^{3}(\mathfrak{g}, \mathfrak{h}; \mathrm{V})$, $$\delta(\Phi)(x,y,z,t) := x.\Phi(y,z,t) -y.\Phi(x,z,t) +z.\Phi(x,y,t) -t.\Phi(x,y,z) -\Phi([x,y],z,t)+\Phi([x,z],y,t)$$
$$-\Phi([x,t],y,z)-\Phi([y,z],x,t)+\Phi([y,t],x,z)-\Phi([z,t],x,y)~~\hbox{for any}~~ x, y, z, t \in \mathfrak{g}.$$
\end{itemize}

%%%%%%%%%%%%%%%%%%%%%%%%%%%%%%%%%%%%%%%%%%%%%%
\section{The space $\mathrm{H}^3_{\mathrm{diff}}\big({\rm Vect}(\mathbb{R}),\mathfrak{aff}(1) ; \mathbb{F}_{\lambda}\big)$}
%%%%%%%%%%%%%%%%%%%%%%%%%%%%%%%%%%%%%%%%%%%%%%%
In this section, we consider the Lie algebra $\mathrm{Vect}(\mathbb{R})$ acting on $\mathbb{F}_{\lambda}$ and we compute the third $\mathfrak{aff}(1)$-relative cohomology space of
$\mathrm{Vect}(\mathbb{R})$ with coefficients in $\mathbb{F}_{\lambda}$. First, we classify $\mathfrak{aff}(1)$-invariant differential operators, then we isolate among them those
that are $3$-cocycles. To do that, we need the following Lemma.
 \begin{lem}\label{inva}
 Any $3$-cocycle $C\in \hbox{Z}^{3}\big(\mathrm{Vect}(\mathbb{R}) ; \mathbb{F}_{\lambda}\big)$ vanishing on the Lie algebra $\mathfrak{aff}(1)$ of  $\mathrm{Vect}(\mathbb{R})$ is $\mathfrak{aff}(1)$-invariant.
\end{lem}
\begin{proof}.

The $3$-cocycle condition reads as follows:
\begin{align*}
(-1)^{C_{\{X,Y,Z,T\}}}C([X, Y], Z, T) -(-1)^{\widetilde{C}_{\{X,Y,Z,T\}}} \mathbb{L}_{X}^{\lambda}(C)(Y, Z, T) + \circlearrowleft (X, Y, Z, T) =\\ C([X, Z], Y, T)+C([Y, T], X, Z)
\end{align*}
for every $X, Y, Z , T \in{\rm Vect}(\mathbb{R}),~\mathrm{where}~\circlearrowleft (X, Y, Z, T)$ denotes the summands obtained from the two written ones by the cyclic permutation of the
symbols $X, Y, Z, T;$
\begin{equation*}
(-1)^{C_{\{X,Y,Z,T\}}}C([X_1, X_2], X_3, X_4)=\left \{
\begin{array}{ll}
(-1)^{i}C([X_1, X_2], X_3, X_4) & {\rm if } ~\exists~i~/~X_i=Y, \\[3mm]~
0& {\rm otherwise; }
\end{array}\right.
\end{equation*}
and
\begin{equation*}
(-1)^{\widetilde{C}_{\{X,Y,Z,T\}}} \mathbb{L}_{X_1}^{\lambda}(C)(X_2, X_3, X_4)=\left \{
\begin{array}{ll}
(-1)^{i} \mathbb{L}_{X_1}^{\lambda}(C)(X_2, X_3, X_4) & {\rm if } ~\exists~i~/~X_i=Y, \\[3mm]~
0& {\rm otherwise. }
\end{array}\right.
\end{equation*}
  Now, if $X \in \mathfrak{aff}(1),$
  then the equation above becomes
 \begin{equation*}
C([X, Y], Z, T) - C([X, Z], Y, T)  + C([X, T], Y, Z)  = \mathbb{L}_{X}^{\lambda}(C)(Y, Z, T).
\end{equation*}
This condition is nothing but the invariance property.
\end{proof}

%%%%%%%%%%%%%%%%%%%%%%%%%%%%%%%%%%%%%%%%%%%%%%%%%%%%%%%%%%%%%%%%%%%%%%%%%%
\subsection{$\mathfrak{aff}(1)$-invariant differential operators}\label{invariant}
%%%%%%%%%%%%%%%%%%%%%%%%%%%%%%%%%%%%%%%%%%%%%%%%%%%%%%%%%%%%%%%%%%%%%%%%%%
As our $3$-cocycles vanish on $\mathfrak{aff}(1)$, we will investigate $\mathfrak{aff}(1)$-invariant skew-symmetric trilinear differential operators that vanish on $\mathfrak{aff}(1).$
\begin{prop}\label{proposition1}
Any skew-symmetric trilinear differential operators $C_{\lambda} :  \wedge^3 {\rm Vect}(\mathbb{R}) \rightarrow \mathbb{F}_{\lambda},$
which are $\mathfrak{aff}(1)$-invariant and vanish on $\mathfrak{aff}(1),$ is as follows :
\begin{equation*}
 C_{n}(X,Y,Z) = \displaystyle \sum_{\substack{i=2 \\ j>i \\ i+j+l=n }}^{[\frac{n-3}{3}]} c_{i,j,n-i-j}
\begin{vmatrix}
f^{i} & g^{i} & h^{i} \\
f^{j} & g^{j} & h^{j} \\
f^{n-i-j} & g^{n-i-j} & h^{n-i-j}
\end{vmatrix} dx^{n-3}
\end{equation*}
for $X= f\frac{d}{ds}$, $Y =g\frac{d}{ds}$, $Z =h\frac{d}{dx}$, $c_{i,j,n-i-j} \in \mathbb{R}$ and $n\in \mathbb{N}+ 9.$ $[k]$ denotes the integer part of $k$:
\end{prop}
\begin{proof}
The generic form of any such a differential operator is \big(here $X= f\frac{d}{ds}$, $Y =g\frac{d}{ds}$, $Z =h\frac{d}{dx} \in \mathrm{Vect}(\mathbb{R})$\big) :
\begin{equation*}
C_{\lambda}(X,Y,Z) = \sum_{i+j+l\leq k} c_{i,j,l}f^{i} g^{j} h^{l} dx^{\lambda}
\end{equation*}
where $c_{i,j,l}=c_{j,l,i}=c_{l,i,j}=-c_{j,i,l}=-c_{i,l,j}=-c_{l,j,i}\in \mathcal{C}^{\infty}(\mathbb{R})$, and $f^{i}$ stands for $\frac{d^{i}f}{dx^{i}}.$
The invariance property with respect to the vector field $X = \frac{d}{dx}$ with arbitrary $Y$ and $Z$ implies that $c'_{i,j,l}=0.$ Therefore $c_{i,j,l}$ are constants. Now, the invariance property with respect to the vector field $X = x \frac{d}{dx}$ with arbitrary $Y$ and $Z$ implies that $i+j+l = \lambda+3,$ so in particular $\lambda$ is integer.
\end{proof}
%%%%%%%%%%%%%%%%%%%%%%%%%%%%%%%%%%%%%%%%%%%%%%%%%%%%%%%%%%%%%%%%%%%%%%%%%%
\subsection{$\mathfrak{aff}(1)$-relative cohomology of $\mathrm{Vect}(\mathbb{R})$}
%%%%%%%%%%%%%%%%%%%%%%%%%%%%%%%%%%%%%%%%%%%%%%%%%%%%%%%%%%%%%%%%%%%%%%%%%%
The main result of this section is the following
\begin{thm} \label{main}
We have
\begin{equation*}
\mathrm{H}^3_{\mathrm{diff}}\big({\rm Vect}(\mathbb{R}),\mathfrak{aff}(1) ; \mathbb{F}_{\lambda}\big) = \left \{
\begin{array}{ll}
\mathbb{R} & {\rm if } ~\lambda=9,11 \\[3mm]~
0& {\rm otherwise. }
\end{array}
\right.
\end{equation*}
The corresponding spaces $\mathrm{H}^3_{\mathrm{diff}}\big({\rm Vect}(\mathbb{R}),\mathfrak{aff}(1) ; \mathbb{F}_{\lambda}\big)$ are spanned by the cohomology classes of the following non-trivial $3$-cocycles :
\begin{equation*}
\Omega_{12}(X, Y, Z) =  \begin{vmatrix}
f^{3} & g^{3} & h^{3} \\
f^{4} & g^{4} & h^{4} \\
f^{5} & g^{5} & h^{5}
\end{vmatrix} dx^{9}
\end{equation*}

\begin{equation*}
\Omega_{14}(X, Y, Z) = \left ( 5 \begin{vmatrix}
f^{3} & g^{3} & h^{3} \\
f^{4} & g^{4} & h^{4} \\
f^{7} & g^{7} & h^{7}
\end{vmatrix}
-14 \begin{vmatrix}
f^{3} & g^{3} & h^{3} \\
f^{5} & g^{5} & h^{5} \\
f^{6} & g^{6} & h^{6}
\end{vmatrix}\right)dx^{11}
\end{equation*}
for $X= f\frac{d}{dx}$, $Y =g\frac{d}{dx}$, $Z =h\frac{d}{dx}$.
\end{thm}

\begin{rmk}
The cohomology space $\mathrm{H}^2_{\mathrm{diff}}\big({\rm Vect}(\mathbb{R}),\mathfrak{aff}(1) ; \mathbb{F}_{\lambda}\big)$ has been computed in \cite{SO}
\end{rmk}
To prove Theorem \ref{main} we need first to study properties of the coboundaries.
\begin{prop}\label{cob}
Every coboundary $\delta(B)\in \mathrm{B}^{3}\big(\mathrm{Vect}(\mathbb{R}), \mathfrak{aff}(1); \mathbb{F}_{n}\big)$ retains the following form
\begin{equation}\label{cob1}
\delta(B)(X,Y,Z)=\sum_{i+j+l=n+3}\beta_{i,j,l}f^{i}g^{j}h^{l}dx^{n}
\end{equation}
 where $$\beta_{0,j,n+3-j}=\beta_{1,j,n+2-j}=0$$
and $B : \mathrm{Vect}(\mathbb{R}) \wedge \mathrm{Vect}(\mathbb{R}) \rightarrow \mathbb{F}_{n}$ be an $\mathfrak{aff}(1)$-invariant operator defined by \big(for $X=f\frac{d}{dx}\in \mathrm{Vect}(\mathbb{R})$ and $Y=g\frac{d}{dx}\in \mathrm{Vect}(\mathbb{R})$\big) :
$$B(X,Y) = \sum_{i+j=n+2}\gamma_{i,j} f^{i}g^{j}dx^{n}$$ where $\gamma_{i,j}=-\gamma_{j,i}$ and $\gamma_{0,n+2}=\gamma_{1,n+1}=0.$
\end{prop}
\begin{proof}
From the very definition of coboundaries, we have
$$\delta(B)(X,Y,Z)=B([X,Y],Z)-B([X,Z],Y)+B([Y,Z],X)-L_{X}^{\lambda}B(Y,Z)+L_{Y}^{\lambda}B(X,Z)-L_{Z}^{\lambda}B(X,Y).$$
The coboundary above vanishes on the Lie algebra $\mathfrak{aff}(1)$. Hence, the operator $B$ is $\mathfrak{aff}(1)$-invariant. The conditions $\gamma_{0,n+2}=\gamma_{1,n+1}=0$ come from the fact that the operator $B$ vanish on $\mathfrak{aff}(1)$. Now, the conditions $\beta_{0,j,n+3-j}=\beta_{1,j,n+2-j}=0$ are consequences of $\mathfrak{aff}(1)$-invariance.
More precisely, any corresponding coboundary is up to a scalar factor as follows
$$
\delta(B)(X,Y,Z) = \displaystyle \sum_{\substack {k=1 \\ i\geq4}}^{[\frac{n-3}{3}]-1} \Bigg[\bigg[\binom{i}{k+1}-\binom{i}{k}\bigg]\gamma_{i,n+2-i}  -  \bigg[\binom{n+1-k}{i-k}-\binom{n+1-k}{i-k-1}\bigg]\gamma_{k+1,n+1-k}$$
\begin{equation} \label{coboundary1}
 + \bigg[\binom{n+2+k-i}{k+1}-\binom{n+2+k-i}{k}\bigg]\gamma_{i-k,n+2+k-i} \Bigg]
\begin{vmatrix}
f^{k+1} & g^{k+1} & h^{k+1} \\
f^{i-k} & g^{i-k} & h^{i-k} \\
f^{n+2-i} & g^{n+2-i} & h^{n+2-i}
\end{vmatrix} dx^{n}
\end{equation}
for $X= f\frac{d}{ds}$, $Y =g\frac{d}{ds}$, $Z =h\frac{d}{dx}$, $c_{i,j,n-i-j} \in \mathbb{R}$, $n\in \mathbb{N} +9$ and where $\binom{n}{k} = \frac{n!}{k!(n-k)!}$.
\end{proof}\vskip0.2cm
{\bf Proof of Theorem \ref{main}}
Let $\Omega_{\lambda}$ be a $3$-cocycle on $\mathrm{Vect}(\mathbb{R})$ vanishing on $\mathfrak{aff}(1)$, with values in $\mathbb{F}_{\lambda}$. By Lemma \ref{inva}, up to a scalar factor, $\Omega_{\lambda}$ is a skew-symmetric trilinear differential operator $\mathfrak{aff}(1)$-invariant. Thus, by Proposition \ref{proposition1} we get the explicit formula for $\Omega_{\lambda}$ :
\begin{equation*}
 C_{n}(X,Y,Z) = \sum_{\substack{i = 2 \\ j>i \\ i+j+l = n}}^{[\frac{n-3}{3}]} c_{i,j,n-i-j}
 \begin{vmatrix}
f^{i} & g^{i} & h^{i} \\
f^{j} & g^{j} & h^{j} \\
f^{n-i-j} & g^{n-i-j} & h^{n-i-j}
\end{vmatrix} dx^{n-3}
\end{equation*}
for $X= f\frac{d}{dx}$, $Y =g\frac{d}{dx}$, $Z =h\frac{d}{dx}$; $c_{i,j,n-i-j} \in \mathbb{R}$ and $n\in \mathbb{N}+ 9.$\\
The $3$-cocycle condition is equivalent to the system (where $2\leq \alpha<\beta<\gamma<\theta$)
$$c_{\alpha+\beta-1,\gamma,\theta}\left[\binom{\alpha+\beta-1}{\alpha}-\binom{\alpha+\beta-1}{\alpha-1}\right]+c_{\beta+\gamma-1,\alpha,\theta}\left[\binom{\beta+\gamma-1}{\beta}-\binom{\beta+\gamma-1}{\beta-1}\right]$$
$$-c_{\alpha+\gamma-1,\beta,\theta}\left[\binom{\alpha+\gamma-1}{\alpha}-\binom{\alpha+\gamma-1}{\alpha-1}\right]+c_{\alpha+\theta-1,\beta,\gamma}\left[\binom{\alpha+\theta-1}{\alpha}-\binom{\alpha+\theta-1}{\alpha-1}\right]$$
\begin{equation} \label{condition}
-c_{\beta+\theta-1,\alpha,\gamma}\left[\binom{\beta+\theta-1}{\beta}-\binom{\beta+\theta-1}{\beta-1}\right]+c_{\gamma+\theta-1,\alpha,\beta}\left[\binom{\gamma+\theta-1}{\gamma}-\binom{\gamma+\theta-1}{\gamma-1}\right]=0.
\end{equation}
This system can be deduce by a simple computation. Of course, such a system has at least one solution in which the solution $c_{i,j,n-i-j}$ are juste the coefficients $\beta_{i,j,n-i-j}$ of the coboundaries (\ref{cob1}).
\begin{itemize}
  \item For $n = 9;$ the operator $C_{9}$ satisfies the $3$-cocycle condition. In this cases, the $3$-cocycle has the forme
   $$\Omega_{12}=(X,Y,Z)=\left[c_{2,3,7}\begin{vmatrix}
f^{2} & g^{2} & h^{2} \\
f^{3} & g^{3} & h^{3} \\
f^{7} & g^{7} & h^{7}
\end{vmatrix}+c_{2,4,6}\begin{vmatrix}
f^{2} & g^{2} & h^{2} \\
f^{4} & g^{4} & h^{4} \\
f^{6} & g^{6} & h^{6}
\end{vmatrix} +c_{3,4,5}\begin{vmatrix}
f^{3} & g^{3} & h^{3} \\
f^{4} & g^{4} & h^{4} \\
f^{5} & g^{5} & h^{5}
\end{vmatrix}\right] dx^{9}.$$

On the other hand, the coboundary (\ref{coboundary1}) takes the form
$$ \delta(B)(X,Y,Z) = (2\gamma_{4,7}-48\gamma_{2,9}+20\gamma_{3,8}) \begin{vmatrix}
f^{2} & g^{2} & h^{2} \\
f^{3} & g^{3} & h^{3} \\
f^{7} & g^{7} & h^{7}
\end{vmatrix} + (5\gamma_{5,6}-42\gamma_{2,9}+14\gamma_{4,7}) \begin{vmatrix}
f^{2} & g^{2} & h^{2} \\
f^{4} & g^{4} & h^{4} \\
f^{6} & g^{6} & h^{6}
\end{vmatrix}$$

The terms $c_{2,3,7}$ and $c_{2,4,6}$ can be eliminated by adding a coboundary with an appropriate value of $\gamma_{4,7}$, $\gamma_{2,9}$ $\gamma_{3,8}$ and $\gamma_{5,6}$  so, we get the $3$-cocycle
   $$\Omega_{12}(X,Y,Z)=c_{3,4,5}\begin{vmatrix}
f^{3} & g^{3} & h^{3} \\
f^{4} & g^{4} & h^{4} \\
f^{5} & g^{5} & h^{5}
\end{vmatrix} dx^{9}.$$
  \item For $n = 10;$ the operator $C_{10}$ is not but a coboundary. Indeed, the $3$-cocycle has the form
  $$\Omega_{13}(X,Y,Z)=\Bigg[c_{2,3,8}\begin{vmatrix}
f^{2} & g^{2} & h^{2} \\
f^{3} & g^{3} & h^{3} \\
f^{8} & g^{8} & h^{8}
\end{vmatrix}+c_{2,4,7}\begin{vmatrix}
f^{2} & g^{2} & h^{2} \\
f^{4} & g^{4} & h^{4} \\
f^{7} & g^{7} & h^{7}
\end{vmatrix} +c_{2,5,6}\begin{vmatrix}
f^{2} & g^{2} & h^{2} \\
f^{5} & g^{5} & h^{5} \\
f^{6} & g^{6} & h^{6}
\end{vmatrix}$$
$$+c_{3,4,6}\begin{vmatrix}
f^{3} & g^{3} & h^{3} \\
f^{4} & g^{4} & h^{4} \\
f^{6} & g^{6} & h^{6}
\end{vmatrix}\Bigg] dx^{10}$$
The $3$-cocycle condition applied to the operator $C_{10}$ leads to the condition $9c_{3,4,6}+14c_{2,3,8}-14c_{2,4,7}+5c_{2,5,6}=0.$ Moreover, the terms $c_{2,3,9}$, $c_{2,4,8}$ and $c_{2,5,7}$ can be eliminated by adding a coboundary and we get a coboundary.
  \item For $n = 11;$ the $3$-cocycle $C_{11}$ has the form
   $$\Omega_{14}(X,Y,Z)=\Bigg[c_{2,3,9}\begin{vmatrix}
f^{2} & g^{2} & h^{2} \\
f^{3} & g^{3} & h^{3} \\
f^{9} & g^{9} & h^{9}
\end{vmatrix}+c_{2,4,8}\begin{vmatrix}
f^{2} & g^{2} & h^{2} \\
f^{4} & g^{4} & h^{4} \\
f^{8} & g^{8} & h^{8}
\end{vmatrix} +c_{2,5,7}\begin{vmatrix}
f^{2} & g^{2} & h^{2} \\
f^{5} & g^{5} & h^{5} \\
f^{7} & g^{7} & h^{7}
\end{vmatrix}$$
$$+c_{3,4,7}\begin{vmatrix}
f^{3} & g^{3} & h^{3} \\
f^{4} & g^{4} & h^{4} \\
f^{7} & g^{7} & h^{7}
\end{vmatrix}+c_{3,5,6}\begin{vmatrix}
f^{3} & g^{3} & h^{3} \\
f^{5} & g^{5} & h^{5} \\
f^{6} & g^{6} & h^{6}
\end{vmatrix}\Bigg] dx^{11}$$
The $3$-cocycle condition applied to the operator $C_{11}$ leads to the condition $5c_{3,5,6}+14c_{3,4,7}-28c_{2,4,8}+42c_{2,3,9}=0.$ Moreover, the terms $c_{2,3,9}$, $c_{2,4,8}$ and $c_{2,5,7} $ can be eliminated by adding a coboundary and we get the cocycle
$$\Omega_{14}= \left [ 5 \begin{vmatrix}
f^{3} & g^{3} & h^{3} \\
f^{4} & g^{4} & h^{4} \\
f^{7} & g^{7} & h^{7}
\end{vmatrix}
-14 \begin{vmatrix}
f^{3} & g^{3} & h^{3} \\
f^{5} & g^{5} & h^{5} \\
f^{6} & g^{6} & h^{6}
\end{vmatrix}\right]dx^{11}.$$
\item For $n = 12,~ 13,~ 14;$\\
Let us show that the solutions of system $(\ref{condition})$ are expressed in terms of $c_{3,s,t}$.\\
In this case $\alpha=2$, the term $c_{2,i,n+1-i}$ can be eliminated as in the foregoing cases then we apply the $3$-cocycle condition, one get by collecting the terms in $f''g^{\beta}h^{\gamma}k^{\theta}$ the conditions :
\begin{equation} \label{condition02}
(\beta+1)(\beta-2)c_{\beta+1,\gamma,\theta}+(\gamma+1)(\gamma-2)c_{\beta,\gamma+1,\theta}+(\theta+1)(\theta-2)c_{\beta,\gamma,\theta+1}=0.
\end{equation}
For $\beta=3$, the equation $(\ref{condition02})$ implies that all the constants $c_{4,i,j}$ can be determined uniquely in terms of $c_{3,s,t}$. More precisely,
$$c_{4,\gamma,\theta}=-\frac{1}{4}\Big[(\gamma+1)(\gamma-2)c_{3,\gamma+1,\theta}+(\theta+1)(\theta-2)c_{3,\gamma,\theta+1}\Big].$$
Using the coboundary expression we eliminate the terms $c_{3,\gamma+1,\theta}$ and $c_{3,\gamma,\theta+1}$, so, we get $c_{4,\gamma,\theta}=0$.\\
Now for $\beta=4$ and from the system $(\ref{condition02})$, we have
$$c_{5,\gamma,\theta}=-\frac{1}{10}\Big[(\gamma+1)(\gamma-2)c_{4,\gamma+1,\theta}+(\theta+1)(\theta-2)c_{4,\gamma,\theta+1}\Big].$$
Hence, we get $c_{5,\gamma,\theta}=0$.\\
By continuing this procedure we see that $c_{6,\gamma,\theta}$, $c_{7,\gamma,\theta}$,... can be determined in terms of $c_{3,s,t}$ which can be eliminated by adding a coboundary. It follows that the cohomology is zero.
\item For $n \geq 15$; we have to study $(\ref{condition})$ for $\alpha=3$, the system has one more equation
$$\bigg[\binom{\beta+2}{3}-\binom{\beta+2}{2}\bigg]c_{\beta+2,\gamma,\theta}+\bigg[\binom{\beta+\gamma-1}{\beta}-\binom{\beta+\gamma-1}{\beta-1}\bigg]c_{\beta+\gamma-1,3,\theta}$$
$$-\bigg[\binom{\gamma+2}{3}-\binom{\gamma+2}{2}\bigg]c_{\gamma+2,\beta,\theta}+\bigg[\binom{\theta+2}{3}-\binom{\theta+2}{2}\bigg]c_{\theta+2,\beta,\gamma}$$
$$-\bigg[\binom{\beta+\theta-1}{\beta}-\binom{\beta+\theta-1}{\beta}\bigg]c_{\beta+\theta-1,3,\gamma}+\bigg[\binom{\gamma+\theta-1}{\gamma}-\binom{\gamma+\theta-1}{\gamma-1}\bigg]c_{\gamma+\theta-1,3,\beta}=0.$$
For $\beta=4$, and from this system we have
$$c_{6,\gamma,\theta}=-\frac{1}{5}\Bigg[\bigg[\binom{\gamma+3}{4}-\binom{\gamma+3}{3}\bigg]c_{3,\gamma+3,\theta}+\bigg[\binom{\gamma+2}{3}-\binom{\gamma+2}{2}\bigg]c_{4,\gamma+2,\theta}$$
$$+\bigg[\binom{\theta+2}{3}-\binom{\theta+2}{2}\bigg]c_{4,\gamma,\theta+2}-\bigg[\binom{\theta+3}{4}-\binom{\theta+3}{2}\bigg]c_{3,\gamma,\theta+3}$$
$$+\bigg[\binom{\gamma+\theta-1}{\gamma}-\binom{\gamma+\theta-1}{\gamma-1}\bigg]c_{3,4,\gamma+\theta-1}\Bigg].$$
By the coboundary expression we eliminate the terms $c_{3,s,t}$ and $c_{4,i,j}$. So, we get $c_{6,\gamma,\theta}=0$.\\
By continuing this procedure we see that $c_{7,\gamma,\theta}$, $c_{8,\gamma,\theta}$,... can be determined in terms $c_{3,s,t}$ and $c_{4,i,j}$ which can be eliminated by adding a coboundary.
 Therefore $$\mathrm{H}^3_{\mathrm{diff}}\big({\rm Vect}(\mathbb{R}),\mathfrak{aff}(1) ; \mathbb{F}_{\lambda}\big)=0$$
\end{itemize}

%%%%%%%%%%%%%%%%%%%%%%%%%%%%%%%%%%%%%%%%%%%%%%%%%%%%%%%%%%%%%%%%%%%%%%%%%%%%
\section{Explicit $3$-cocycles of $\mathrm{Vect}(\mathbb{R})$}
%%%%%%%%%%%%%%%%%%%%%%%%%%%%%%%%%%%%%%%%%%%%%%%%%%%%%%%%%%%%%%%%%%%%%%%%%%%%
The main result of this section is the following
\begin{thm} \label{main2}
We have
\begin{equation*}
\mathrm{H}^3_{\mathrm{diff}}\left({\rm Vect}(\mathbb{R}) ; \mathbb{F}_{\lambda}\right) = \left \{
\begin{array}{lll}
\mathbb{R} & {\rm if } ~\lambda=5,6,9,11; \\[3mm]
\mathbb{R}^{2} & {\rm if } ~\lambda=7,8; \\[3mm]
0& {\rm otherwise. }
\end{array}
\right.
\end{equation*}
\end{thm}
\begin{proof}
Consider a general form of $3$-cocycle \big(where $X= f\frac{d}{ds}$, $Y =g\frac{d}{ds}$ and $Z =h\frac{d}{dx} \in \rm Vect(\mathbb{R})$\big)
$$C(X,Y,Z) = \sum_{\substack{ i+j+l = n+3}} c_{i,j,n+3-i-j} f^{i} g^{j} h^{n+3-i-j} dx^{\lambda},$$
where $c_{i,j,l}=c_{j,l,i}=c_{l,i,j}=-c_{j,i,l}=-c_{i,l,j}=-c_{l,j,i}$.\\
We will eliminate coboundaries in ordre to turn the 3-cocycle above into 3-cocycle vanishing on  $\mathfrak{aff}(1)$. Let $B : \mathrm{Vect}(\mathbb{R}) \wedge \mathrm{Vect}(\mathbb{R}) \rightarrow \mathbb{F}_{n}$ be an operator defined by \big(for $X=f\frac{d}{dx}\in \mathrm{Vect}(\mathbb{R})$ and $Y=g\frac{d}{dx}\in \mathrm{Vect}(\mathbb{R})$\big) :
$$B(X,Y) = \sum_{i+j=n+2}\beta_{i,j} f^{i}g^{j}dx^{n}$$ where $\beta_{i,j}=-\beta_{j,i}.$\\ Consider a general expression of a coboundary
$$\delta B(X,Y,Z)= -\Big[(n+1)\beta_{0,n+2}+\beta_{1,n+1}\Big] \begin{vmatrix}
f & g & h \\
f' & g' & h' \\
f^{n+2} & g^{n+2} & h^{n+2}
\end{vmatrix} dx^{n}$$
$$-\sum_{i \geq 2 } \left[\binom{n+2}{i}-\binom{n+2}{i-1}\right] \beta_{0,n+2}
 \begin{vmatrix}
f & g & h \\
f^{i} & g^{i} & h^{i} \\
f^{n+3-i} & g^{n+3-i} & h^{n+3-i}
\end{vmatrix} dx^{n}$$
$$-\sum_{i \geq 2} \left[\binom{n+1}{i}-\binom{n+1}{i-1}\right] \beta_{1,n+1}
 \begin{vmatrix}
f' & g' & h' \\
f^{i} & g^{i} & h^{i} \\
f^{n+2-i} & g^{n+2-i} & h^{n+2-i}
\end{vmatrix} dx^{n} + \hbox{higher order terms.}$$
Immediately we see that the constant $c_{0,1,n+2}$ can be eliminated upon putting $c_{0,1,n+2}=-\big[(n+1)\beta_{0,n+2}+\beta_{1,n+1}\big]$. On the other hand, the $3$-cocycle condition implies that $c_{0,i,n+3-i}=-\frac{1}{n}\left[\binom{n+2}{i}-\binom{n+2}{i-1}\right]c_{0,1,n+2}$.\\

$\bullet$ For $n=3,$ the $3$-cocycle takes the form
$$\Omega_{6}(X, Y, Z) = c_{1,2,3} \begin{vmatrix}
f^{1} & g^{1} & h^{1} \\
f^{2} & g^{2} & h^{2} \\
f^{3} & g^{3} & h^{3}
\end{vmatrix} dx^{3}.$$
On the other hand, the coboundary takes the form
$$\delta B(X,Y,Z)= 2 \beta_{1,4}\begin{vmatrix}
f^{1} & g^{1} & h^{1} \\
f^{2} & g^{2} & h^{2} \\
f^{3} & g^{3} & h^{3}
\end{vmatrix} dx^{3}$$
where $\beta_{1,4}$ is a constant. So, the terms $c_{1,2,3}$ can be eliminated by adding the coboundary.

$\bullet$ For $n=4,$ the $3$-cocycle takes the form
$$\Omega_{7}(X, Y, Z) = c_{1,2,4} \begin{vmatrix}
f^{1} & g^{1} & h^{1} \\
f^{2} & g^{2} & h^{2} \\
f^{4} & g^{4} & h^{4}
\end{vmatrix} dx^{3}.$$
On the other hand, the coboundary takes the form
$$\delta B(X,Y,Z)= 5 \beta_{1,5}\begin{vmatrix}
f^{1} & g^{1} & h^{1} \\
f^{2} & g^{2} & h^{2} \\
f^{4} & g^{4} & h^{4}
\end{vmatrix} dx^{4}$$
where $\beta_{1,5}$ is a constant. So, the terms $c_{1,2,4}$ can be eliminated by adding the coboundary.

$\bullet$ For $n=5,$ the $3$-cocycle takes the form
$$\Omega_{8}(X, Y, Z) = \left[c_{1,2,5} \begin{vmatrix}
f^{1} & g^{1} & h^{1} \\
f^{2} & g^{2} & h^{2} \\
f^{5} & g^{5} & h^{5}
\end{vmatrix} + c_{1,3,4} \begin{vmatrix}
f^{1} & g^{1} & h^{1} \\
f^{3} & g^{3} & h^{3} \\
f^{4} & g^{4} & h^{4}
\end{vmatrix}\right]dx^{5}.$$
On the other hand, the coboundary takes the form
$$\delta B(X,Y,Z)= \left[9 \beta_{1,6}\begin{vmatrix}
f^{1} & g^{1} & h^{1} \\
f^{2} & g^{2} & h^{2} \\
f^{5} & g^{5} & h^{5}
\end{vmatrix} +5 \beta_{1,6}\begin{vmatrix}
f^{1} & g^{1} & h^{1} \\
f^{3} & g^{3} & h^{3} \\
f^{4} & g^{4} & h^{4}
\end{vmatrix} \right]dx^{5}$$
where $\beta_{1,6}$ is a constant. So, the terms $c_{1,2,5}$ or $c_{1,3,4}$ can be eliminated by adding a coboundary. Hence, the cohomology group is one-dimensional.

$\bullet$ For $n=6,$ the $3$-cocycle takes the form
$$\Omega_{9}(X, Y, Z) =\left[ c_{1,2,6} \begin{vmatrix}
f^{1} & g^{1} & h^{1} \\
f^{2} & g^{2} & h^{2} \\
f^{6} & g^{6} & h^{6}
\end{vmatrix} + c_{1,3,5} \begin{vmatrix}
f^{1} & g^{1} & h^{1} \\
f^{3} & g^{3} & h^{3} \\
f^{5} & g^{5} & h^{5}
\end{vmatrix}+ c_{2,3,4} \begin{vmatrix}
f^{2} & g^{2} & h^{2} \\
f^{3} & g^{3} & h^{3} \\
f^{4} & g^{4} & h^{4}
\end{vmatrix}\right]dx^{6}.$$
The $3$-cocycle condition gives $c_{2,3,4}+c_{1,2,6}-c_{1,3,5}=0$ and the terms $c_{1,2,6}$ or $c_{1,3,5}$ can be eliminated by adding a coboundary. Hence, the cohomology group is one-dimensional.

$\bullet$ For $n=7,$ the $3$-cocycle takes the form
$$\Omega_{10}(X, Y, Z) = \Bigg[c_{1,2,7} \begin{vmatrix}
f^{1} & g^{1} & h^{1} \\
f^{2} & g^{2} & h^{2} \\
f^{7} & g^{7} & h^{7}
\end{vmatrix} + c_{1,3,6} \begin{vmatrix}
f^{1} & g^{1} & h^{1} \\
f^{3} & g^{3} & h^{3} \\
f^{6} & g^{6} & h^{6}
\end{vmatrix}+c_{1,4,5} \begin{vmatrix}
f^{1} & g^{1} & h^{1} \\
f^{4} & g^{4} & h^{4} \\
f^{5} & g^{5} & h^{5}
\end{vmatrix}$$
 $$+c_{2,3,5} \begin{vmatrix}
f^{2} & g^{2} & h^{2} \\
f^{3} & g^{3} & h^{3} \\
f^{5} & g^{5} & h^{5}
\end{vmatrix}\Bigg]dx^{7}.$$
The $3$-cocycle condition gives $7c_{2,3,5}-2c_{1,4,5}+14c_{1,2,7,}-9c_{1,3,6}=0$ and one of the terms $c_{1,2,7}$, $c_{1,3,6}$ or $c_{1,4,5}$ can be eliminated by adding the coboundary, Hence, the cohomology group is two-dimensional.

$\bullet$ For $n=8,$ the $3$-cocycle takes the form
$$\Omega_{11}(X, Y, Z) = \Bigg[c_{1,2,8} \begin{vmatrix}
f^{1} & g^{1} & h^{1} \\
f^{2} & g^{2} & h^{2} \\
f^{8} & g^{8} & h^{8}
\end{vmatrix} + c_{1,3,7} \begin{vmatrix}
f^{1} & g^{1} & h^{1} \\
f^{3} & g^{3} & h^{3} \\
f^{7} & g^{7} & h^{7}
\end{vmatrix}+c_{1,4,6} \begin{vmatrix}
f^{1} & g^{1} & h^{1} \\
f^{4} & g^{4} & h^{4} \\
f^{6} & g^{6} & h^{6}
\end{vmatrix}$$
 $$+c_{2,3,6} \begin{vmatrix}
f^{2} & g^{2} & h^{2} \\
f^{3} & g^{3} & h^{3} \\
f^{6} & g^{6} & h^{6}
\end{vmatrix} +c_{2,4,5} \begin{vmatrix}
f^{2} & g^{2} & h^{2} \\
f^{4} & g^{4} & h^{4} \\
f^{5} & g^{5} & h^{5}
\end{vmatrix} \Bigg]dx^{8}.$$
The $3$-cocycle condition gives $9c_{2,3,6}+28c_{1,2,8}-14c_{1,3,7,}-2c_{1,4,6}=0$, $8c_{2,4,5}+14c_{1,2,8}-9c_{1,4,6}=0$ and one of the terms $c_{1,2,8}$, $c_{1,3,7}$ or $c_{1,4,6}$ can be eliminated by adding the coboundary, Hence, the cohomology group is two-dimensional..\\
$\bullet$ Suppose now that $n>8$. We will deal with the coefficients $c_{1,\gamma,n+2-\gamma}$. The $3$-cocycle condition implies that the component of $f'g^{\beta}h^{\gamma}k^{n+3-\beta-\gamma}$, which should be zero, is equal to
$$c_{\beta+\gamma-1,1,n+3-\beta-\gamma}\left[\binom{\beta+\gamma-1}{\beta}-\binom{\beta+\gamma-1}{\beta-1}\right]-c_{n+2-\gamma,1,\gamma}\left[\binom{n+2-\gamma}{\beta}-\binom{n+2-\gamma}{\beta-1}\right]$$
\begin{equation} \label{equation33}
 +c_{n+2-\beta,1,\beta}\left[\binom{n+2-\beta}{\gamma}-\binom{n+2-\beta}{\gamma-1}\right].\end{equation}
The coefficient of $f'g''h^{n}$ is zero in the expression of the coboundary upon putting $c_{1,2,n}=-\frac{1}{2}(n+1)(n-2)\beta_{1,n+1}$. But $c_{1,3,n-1}$ can be eliminated upon putting $c_{1,3,n-1}=-\frac{1}{6}n(n+1)(n-4)\beta_{1,n+1}$. By putting $\beta=2$, we can see from (\ref{equation33}) that all $c_{1,t,n+2-t}$ can be expressed in terms of $c_{1,2,n}$. They are given by the induction formula :
\begin{equation*} c_{1,s,n+2-s}=\frac{2}{s(s-3)}\Bigg[-c_{1,s-1,n+3-s}\left[\binom{n+3-s}{2}-\binom{n+3-s}{1}\right]+c_{1,2,n}\left[\binom{n}{s-1}-\binom{n}{s-2}\right]\Bigg] \end{equation*}
for $s>3.$\\
However, for $\beta=3$ and $\gamma=4$, the system (\ref{equation33}) become :
$$\binom{n}{4}(84+n(n-4)(n-7))c_{1,2,n}=0.$$
Since $n>4$, the equation above admits a solution only for $c_{1,2,n}=0$. Thus, all $c_{1,\gamma,n+2-\gamma}$ are zero.
Finally, the remaining 3-cocycle vanishes on $\mathfrak{aff}(1)$.
\end{proof}

\begin{lemma}
Every coboundary $\delta (B) \in \mathrm{B}^{3}\left(\mathrm{Vect}(\mathbb{R}) ; \mathbb{F}_{n}\right)$ vanishing on $\mathfrak{aff}(1)$ retains the following form
$$\delta(B)(X,Y,Z)=\sum_{i+j+l=n+3}\beta_{i,j,l}f^{i}g^{j}h^{l}dx^{n}$$ where $$\beta_{0,j,n+3-j}=\beta_{1,j,n+2-j}=0.$$
So any corresponding coboundary is up to a scalar factor as follows
$$\delta(B)(X,Y,Z) = \displaystyle \sum_{\substack {k=1 \\ i\geq4}}^{[\frac{n-6}{3}]} \Bigg[\bigg[\binom{i}{k+1}-\binom{i}{k}\bigg]\gamma_{i,n+2-i}  -  \bigg[\binom{n+1-k}{i-k}-\binom{n+1-k}{i-k-1}\bigg]\gamma_{k+1,n+1-k}$$
\begin{equation*} \label{coboundary}
 + \bigg[\binom{n+2+k-i}{k+1}-\binom{n+2+k-i}{k}\bigg]\gamma_{i-k,n+2+k-i} \Bigg]
\begin{vmatrix}
f^{k+1} & g^{k+1} & h^{k+1} \\
f^{i-k} & g^{i-k} & h^{i-k} \\
f^{n+2-i} & g^{n+2-i} & h^{n+2-i}
\end{vmatrix} dx^{n-3}.
\end{equation*}
\end{lemma}
\begin{proof}
Similar to Theorem \ref{main}.
\end{proof}

%%%%%%%%%%%%%%%%%%%%%%%%%%%%%%%%%%%%%%%%%%%%%%%%%%%%%%%%%%%%%%%%%%%%%%%%%%%%

\end{document}